\documentclass[12pt]{amsart}

\usepackage{amsmath}
\usepackage{amsfonts}
\usepackage{amsthm}
\usepackage{amssymb}
\usepackage{enumerate}
\usepackage{tikz-cd}
\usepackage{graphicx}
\usepackage{footmisc}
\usepackage{enumerate}
\usepackage{centernot}
\usepackage{mathtools}
\usepackage{stmaryrd}
\usepackage{url}
\usepackage{multicol}
\usepackage{pdfpages}
\usepackage{colortbl}

\usepackage[style=numeric]{biblatex}
\newtheorem{definition}{Definition}[section]

\newtheorem{corollary}[definition]{Corollary}
\newtheorem{proposition}[definition]{Proposition}

\makeatletter
\newcommand{\newreptheorem}[2]{\newtheorem*{rep@#1}{\rep@title}\newenvironment{rep#1}[1]{\def\rep@title{#2 \ref*{##1}}\begin{rep@#1}}{\end{rep@#1}}}
\makeatother

\theoremstyle{definition}

\newtheorem{remark}[definition]{Remark}

\newcommand{\gal}{\mathrm{Gal}}

\newcommand{\im}{\mathrm{im}}
\newcommand{\cl}{\mathrm{Cl}}

\newcommand{\Ep}{E^{(p)}}
\newcommand{\leg}[2]{\genfrac{(}{)}{}{}{#1}{#2}}

    \DeclareFontFamily{U}{wncy}{}
    \DeclareFontShape{U}{wncy}{m}{n}{<->wncyr10}{}
    \DeclareSymbolFont{mcy}{U}{wncy}{m}{n}
    \DeclareMathSymbol{\Sh}{\mathord}{mcy}{"58}
    
\usepackage[utf8]{inputenc}
    
\title{A remark on prime (non)congruent numbers}

\author[T.~Evink]{Tim Evink}
\address{Institute of Algebra and Number Theory, Ulm University, 
Helmholtzstr.~18, 89081 Ulm, Germany.}
\email{tim.evink@uni-ulm.de}
\author[J.~Top]{Jaap Top}
\address{Bernoulli Institute, University of Groningen,
Nijenborgh~9, 9747~AG~Groningen, the Netherlands.}
\email{j.top@rug.nl}
\author[J.D.~Top]{Jakob Dirk Top}
\address{Bernoulli Institute, University of Groningen,
Nijenborgh~9, 9747~AG~Groningen, the Netherlands.}
\email{j.d.top@rug.nl}
 
\begin{document}

\maketitle

\begin{abstract}
     This paper discusses prime numbers that are (resp. are not)
     congruent numbers. Particularly the only case not fully covered by
     earlier results, namely primes of the form
     $p=8k+1$, receives attention.
 \end{abstract}

\section{Introduction}

A congruent number is, by definition, a positive integer $n$
that occurs as the area of a rectangular triangle with rational
side lengths. In other words, nonzero $a,b,c\in\mathbb{Q}$ should
exist such that $a^2+b^2=c^2$ and $ab/2=n$.
A classical, equivalent definition is that $n>0$ is congruent
precisely when an arithmetic progression $(x-n,x,x+n)$ consisting
of three rational squares exists. 
There is an abundance of literature on congruent numbers,
particularly because of their connection to the arithmetic
of elliptic curves and to modular forms. For example, the textbook
\cite{Koblitz} describes many of the spectacular results on
congruent numbers. Part of it is also discussed in \cite{TopYui},
and various centuries old results including some that already
appeared in Leonardo Pisano's {\sl Liber Quadratorum}
published in 1225, can be found in \cite[Chapter~XVI]{Dickson}. Many investigations regarding
(non)congruent numbers start with the standard and well-known
observation $(1)\Leftrightarrow(2)\Leftrightarrow(3)$, with
\begin{enumerate}
\item[(1)] $n$ is congruent;
\item[(2)] $E^{(n)}\colon y^2=x^3-n^2x$ contains a point $(a,b)\in E^{(n)}(\mathbb{Q})$ with $b\neq 0$;
\item[(3)] $E^{(n)}\colon y^2=x^3-n^2x$ contains $P\in E^{(n)}(\mathbb{Q})$ of infinite order.
\end{enumerate}

In the present note we collect some information regarding
prime numbers that are congruent numbers or noncongruent numbers. Concretely:
\begin{itemize}
    \item We review known results, including
    detailed references to in some cases very
    classical sources (the remainder of the present section).
    \item We recall the equivalence of
    various conditions that imply a prime $p\equiv 1\bmod 8$ to be noncongruent
    and reformulate this in a way allowing one
    to obtain a density result on the number of such primes (Proposition~\ref{equiv1mod8} and
    Corollary~\ref{Densityresult}).
    \item By comparing a $2$-descent over $\mathbb{Q}$ with a $2$-descent over
    $\mathbb{Q}(\sqrt{p})$ we obtain a new
    and relatively general approach towards
    obtaining classical results on a class
    of primes $p\equiv 1\bmod 8$ not being congruent (Propositions~\ref{SelQ} and \ref{SelK} as well as Corollary~\ref{Densityresult}).
    \item The only primes for which no unconditional general statement regarding congruence is known to date,
    are the $p\equiv 1\bmod 8$ ones
    not in the class referred to above. We present some data
    useful for predicting density
    statements in this case (Section~\ref{Sect3}).
    \item We show how a conjecture of Bouniakowsky
    implies the existence of infinitely
    many primes $p\equiv 1\bmod 8$ that are congruent numbers. Moreover we estimate
    the number of primes below any bound $b$ that our construction is predicted to give (Proposition~\ref{examples1mod8} and Remark~\ref{DensityRemark}).
\end{itemize}
 Recognizing whether a prime is congruent or not, is also the
topic of the 2006 master's thesis \cite{Hemenway}. The detailed
explanation given there allows us to remain brief on some issues here.

\vspace{\baselineskip}
Historically,  $p=2$ is the subject of
\cite[Prop.~XII, pp.~114--116]{Barlow}: Peter Barlow already in 1811
showed that $2$ is not congruent. As for the odd
primes $p$, probably the earliest result not restricted to
a single prime is due to A.~Genocchi (1855, 
\cite[pp.~314--315]{Genocchi}, see also T.~Nagell's 1929 expository text \cite[pp.~16--17]{Nagell}): no prime $p\equiv 3\bmod 8$ is congruent.
The same result with a totally different and much less elementary proof is presented
in \cite[Prop.~5]{Tunnell}.
In contrast, {\em every} prime $p\equiv 5,7\bmod 8$ turns out
to be congruent. This is observed in a short paper by
N.M.~Stephens \cite[bottom of p.~183]{Stephens}; a detailed
proof can be found in \cite{Monsky}.

The situation for $p\equiv 1\bmod 8$ is less complete.
L.~Bastien \cite{Bastien} in 1915 announced that, writing
$p=a^2+b^2$ as a sum of two squares, if the Legendre symbol
$\leg{a+b}{p}$ equals $-1$ then $p$ is noncongruent.
The same result can be deduced from a theorem by Michael J. Razar (1974): although his condition on $p$ looks different at first sight and
the notion ``congruent number'' does not occur in the paper,
from \cite[Thm~2]{Razar} and the remarks on ``first descent''
contained in the text it is immediate that if one of the 
following conditions holds:
\begin{itemize}
\item[(a)] $p\equiv 1\bmod 16$ and $2$ is not a $4$th power modulo $p$; 
\item[(b)] $p\equiv 9\bmod 16$ and $2$ is a $4$th power modulo $p$,
\end{itemize}
then $p$ is not a congruent number.
Yet another way to formulate this result, is presented in
J.B.~Tunnell's 1983 paper \cite[Prop.~6]{Tunnell}: if the prime
$p\equiv 1\bmod 8$ is written as $p=a^2+4b^2$ for integers $a,b$,
then $16\nmid p-1+4b\Rightarrow p\;\text{is not congruent}$.
Tunnell's proof is very different from Razar's. Again an alternative way
to formulate and prove the same result one 
finds in the 2006 master's
thesis of Brett Hemenway \cite[p.~41]{Hemenway}. He shows, 
writing $p\equiv 1\bmod 8$ as $p=a^2+b^2$ for integers $a,b$, that
if $(a+b)^2\equiv 9\bmod 16$ then $p$ is not congruent. His
proof (in fact similar to Razar's but more detailed) consists of
a $2$-descent on the elliptic curve given by $y^2=x^3+4p^2x$;
this curve is $2$-isogenous to the one with equation $y^2=x^3-p^2x$.
Although involving different curves, this is very much in the
spirit of \cite{Stroeker-Top}.
The last way we mention of formulating
an equivalent result is a
special case of
\cite[Thm.~1.1(1)]{Ouyang-Zhang}: if
for a prime $p\equiv 1\bmod 8$ the
Legendre symbol $\leg{1+i}{p}=-1$ (here
$i\in\mathbb{F}_p$ is a square root of $-1$), then $p$ is noncongruent.

In the next section we discuss the equivalence of the criteria above.
In particular this allows one to find the density of the involved set of primes. We also sketch how, apart from the fact that all primes 
$\equiv 5,7\bmod p$ are congruent, the results mentioned here
can be verified by comparing $2$-descents over $\mathbb{Q}$ and
over $\mathbb{Q}(\sqrt{p})$; this unified approach,
inspired by \cite{EvdHT} and in particular by Propositions~5.5-5.6
in {\sl loc. sit.}, is somewhat
different from what one finds in earlier literature.
The last section contains results and some data regarding primes $p\equiv 1\bmod 8$ {\em not} satisfying the equivalent conditions described above.
\section{Equivalences and Selmer groups}
Consider a prime number $p\equiv 1\bmod 8$. The congruence condition
implies that $-1$ is a fourth power modulo $p$ and that $2$ is
a square modulo $p$. Moreover, one can write $p=a^2+b^2$ for
integers $a,b$ of which one is odd and the other is divisible by $4$.
The equivalences mentioned in the introduction are included in
the next result. It uses a certain Galois extension $N/\mathbb{Q}$,
namely the splitting field over $\mathbb{Q}$ of the minimal
polynomial $x^4-2x^2+2$ of $\alpha:=\sqrt{1+i}$; this minimal
polynomial has zeroes $\pm\alpha$ and $\pm\beta$ with $(\alpha\beta)^2=2$. Hence $N=\mathbb{Q}(\sqrt{2},\alpha)$ and
the Galois group $\gal(N/\mathbb{Q})$ is the dihedral group of
order $8$, with generators $\rho$ (of order $4$) defined by
$\alpha\mapsto \beta\mapsto -\alpha$ and $\sigma$ (of order $2$)
defined by $\alpha\mapsto\beta\mapsto\alpha$. Note
$\gal(N/\mathbb{Q}(\sqrt{-2}))=\langle\rho\rangle$ (cyclic group of order $4$; $(\alpha^3-\alpha)\beta$ is a square root of $-2$ and is fixed by $\rho$). The element $\rho^2$ is the unique nontrivial
element in the center of
$\gal(N/\mathbb{Q})$; the field of invariants under $\rho^2$
equals $\mathbb{Q}(i,\sqrt{2})$ which is the $8$-th cyclotomic field.
In particular, $\rho^2$ being in the center implies that a
Frobenius element as appearing in part (7) of the next result,
is well-defined (and not only defined up to conjugacy).
\begin{proposition}
\label{equiv1mod8}
For a prime $p\equiv 1\bmod 8$ the following statements are equivalent.
\begin{enumerate}
    \item[{\rm(1.)}] For $j\in\mathbb{Z}$ with $j^2\equiv -1\bmod p$ one has
    $\leg{1+j}{p}=-1$;
    \item[{\rm(2.)}] For $r\in\mathbb{Z}$ with $r^2\equiv 2\bmod p$ one has
    $\leg{1+r}{p}=-1$;
    \item[{\rm(3.)}] For $r\in\mathbb{Z}$ with $r^2\equiv 2\bmod p$ one has
    $(-1)^{(p-1)/8}\cdot\leg{r}{p}=-1$;
    \item[{\rm(4.)}] Writing $p=a^2+b^2$ with $a,b\in\mathbb{Z}$, one has
    $\leg{a+b}{p}=-1$;
    \item[{\rm(5.)}] Writing $p=a^2+b^2$ with $a,b\in\mathbb{Z}$, one has
    $(a+b)^2\equiv 9\bmod 16$;
    \item[{\rm(6.)}] Writing $p=a^2+4c^2$ with $a,c\in\mathbb{Z}$, one has
    $16\nmid p-1+4c$;
    \item[{\rm(7.)}] The (unique) Frobenius element in $G:=\gal(\mathbb{Q}(\sqrt{2},\sqrt{1+i})/\mathbb{Q})$
    corresponding to $p$ equals the generator of the
    center of $G$.
    \item[{\rm(8.)}] With $M$ the normal closure of $\mathbb{Q}(\sqrt{1+\sqrt{2}})$ over $\mathbb{Q}$ which is a dihedral extension of
    degree $8$, the (unique) Frobenius element in $\gal(M/\mathbb{Q})$
    corresponding to $p$ equals the generator of the
    center.
\end{enumerate}
\end{proposition}
\begin{proof}
(1.)$\Leftrightarrow$(2.): this is also shown in \cite[Lemma~2]{Razar}. If $i,r\in\mathbb{F}_p$ satisfy $i^2=-1$ and $r^2=2$, then one has $(1+i)(1+r)=(1+\zeta)^2$ for $\zeta\in\mathbb{F}_p^\times$ of order $8$. The result follows.\\
(1.)$\Leftrightarrow$(3.): with $\zeta\in\mathbb{F}_p$ as
above, note that $i:=\zeta^2$ satisfies $i^2=-1$ and
$r:=\zeta\cdot(1-i)$ has the property $r^2=2$. Therefore
\[
(-1)^{(p-1)/8}\cdot\leg{r}{p}=\leg{\zeta}{p}\cdot\leg{r}{p}=
\leg{\zeta^2\cdot (1-i)}{p}=\leg{1+i}{p}.
\]
(1.)$\Leftrightarrow$(4.): this is also shown in \cite[\S5.4.1]{Hemenway}, for completeness we recall the proof.
In $p=a^2+b^2$ we can and will assume $a$ is odd. Using the
Jacobi symbol and in particular reciprocity for it, one finds the property
$\leg{a}{p}=\leg{p}{a}=\leg{a^2+b^2}{a}=1$, hence with
$i:=(b\bmod p)/(a\bmod p)\in\mathbb{F}_p$ a primitive $4$th root of unity $\leg{a+b}{p}=\leg{a}{p}\cdot\leg{1+i}{p}=\leg{1+i}{p}$.\\
(4.)$\Leftrightarrow$(5.): this also follows using
\cite[\S5.4.1]{Hemenway}, we repeat the calculation.
\[
\leg{a\!+\!b}{p}=\leg{p}{a\!+\!b}=\leg{2}{a\!+\!b}\leg{(a\!+\!b)^2+(a\!-\!b)^2}{a+b}=(-1)^{\left((a+b)^2-1)\right)/8},
\]
from which the result is immediate.\\
(5.)$\Leftrightarrow$(6.): as before, write
$p=a^2+b^2$ with $a$ odd. The assumption $p\equiv 1\bmod 8$ implies $4|b$. With $b=2c=4d$ one obtains
congruences
$(a+b)^2\equiv a^2+8ad\equiv a^2+8d\equiv p+4c\bmod 16$ and the result follows.\\
(1.)$\Leftrightarrow$(7.): using the notations preceding the statement of Prop.~\ref{equiv1mod8}, one has
$\mathbb{Q}\subset\mathbb{Q}(i)\subset
\mathbb{Q}(\zeta)\subset
\mathbb{Q}(\zeta,\alpha)=N$ where $\zeta$ is
a primitive $8$th root of unity, and each consecutive
extension has degree $2$. Note that $\mathbb{Q}(\zeta)$ is the field of invariants of
the generator $\rho^2$ of the center of $\gal(N/\mathbb{Q})$.
The condition $p\equiv 1\bmod 8$ is equivalent to
$p$ splitting completely in $\mathbb{Q}(\zeta)$.
Hence $p$-adically one obtains
$\mathbb{Q}_p(\zeta)=\mathbb{Q}_p$ with the extension obtained by adjoining a square root of
$1+i\in\mathbb{Q}_p$. This extension is nontrivial
precisely when a Frobenius at $p$ acts
nontrivially on $N$, hence as $\rho^2$. Since
$1+i$ is a square in $\mathbb{Q}_p$ iff $\leg{1+i}{p}=1$, this proves the result.\\
(2.)$\Leftrightarrow$(8.): the minimal polynomial of $\gamma:=\sqrt{1+\sqrt{2}}$ over $\mathbb{Q}$ equals
$x^4-2x^2-1$ with zeroes $\pm\gamma$ and $\pm\delta$,
where $(\gamma\delta)^2=-1$. The splitting field
$M$ fits in a tower
$\mathbb{Q}\subset\mathbb{Q}(i)\subset
\mathbb{Q}(\zeta)\subset
\mathbb{Q}(\zeta,\gamma)=M$ where $\zeta$ is
a primitive $8$th root of unity and each consecutive
extension has degree $2$. The remainder of the
argument is analogous to the proof of (1.)$\Leftrightarrow$(7.), with now $1+\sqrt{2}$
taking the role of
$1+i$.
\end{proof}
\begin{remark}
Note that the fields $N$ and $M$ appearing here
are isomorphic; this provides a direct proof of 
(7.)$\Leftrightarrow$(8.). This elementary observation is in fact the special case
$(Q,a,b,x,y,z)=(\mathbb{Q},-1,2,1,1,1)$ of
\cite[Cor.~5.2]{stevenhagen2018redei}.

Several of the conditions appearing in Prop.~\ref{equiv1mod8} play a role in
various papers on $8$-rank of class groups. For example,
\cite[Thm.~2]{Kaplan} states for primes $p\equiv 1\bmod 8$ that
the class number of $\mathbb{Q}(\sqrt{-p})$ is not divisible by $8$
precisely when condition (5.) holds. This equivalence was already
observed by Gauss in a letter to Dirichlet dated
30~May 1828. Similarly, reformulating
\cite[Thm.~1]{Brown} yields that for primes $p\equiv 1\bmod 8$ 
the class number of $\mathbb{Q}(\sqrt{-p})$ is not divisible by $8$
precisely when condition (6.) holds. Included in
\cite[Main Theorem]{BarrucandCohn} as well as in part of \cite[Thm~1]{Stevenhagen93}, the same class number
criterion is shown to be equivalent to (1.). A reason for these relations with
class numbers can be found as a special
case of \cite[Thm.~1.1]{Wang}. Still other equivalent criteria as
well as many of the ones above are listed in \cite[Lemma-Def.~1]{Li-Tian} and \cite[Thm~4.2]{Wang}. As an
additional, related and very accessible text in the same spirit
we refer to \cite{BruinHem}.

Alternatively, the conditions (1.)--(8.) can  be expressed in terms of a
R\'edei symbol (see, e.g., \cite[Section~6]{stevenhagen2018redei}).
They are equivalent to the statement
$[-1,2,p]=-1$.
\end{remark}

We now provide some details regarding the $2$-descent calculations
for the elliptic curves involved. For any prime number $p$, write
\[
\Ep\colon y^2=x^3-p^2x.
\]
The curve $\Ep/\mathbb{Q}$ has good reduction away from $\{2,p\}$.
With notations as in, e.g., \cite[Ch.~X \S1]{Silverman} take
$S=\{2,p,\infty\}$ and denote $\mathbb{Q}(S)=\mathbb{Q}(S,2):=\{x\in\mathbb{Q}^*/{\mathbb{Q}^*}^2\;:\;
v(x)\equiv 0\bmod 2\;\text{for all}\;v\not\in S\}$
which is an elementary $2$-group generated by
the classes of $-1, 2$, and $p$. With
\[
H:= \left\{ (c_1, c_2, c_3)\in \mathbb{Q}(S)\times \mathbb{Q}(S)\times\mathbb{Q}(S)\;:\; c_1c_2c_3=1
\right\},
\]
one has the Kummer homomorphism
$\delta\colon \Ep(\mathbb{Q})\to H$ with kernel
$2\Ep(\mathbb{Q})$ given by 
\[\delta(a,b)=\left\{\begin{array}{ll}
(a+p,a,a-p) & \text{for } b\neq 0;\\
(2,-p,-2p) & \text{for } a=-p;\\
(p,-1,-p) & \text{for } a=0;\\
(2p,p,2) & \text{for } a=p.
\end{array}\right.
\]
One has local versions
$\delta_v\colon \Ep(\mathbb{Q}_v)\to H_v$ of this,
obtained by replacing $\mathbb{Q}$ by
the completion $\mathbb{Q}_v$. The inclusion $\mathbb{Q}\subset\mathbb{Q}_v$ induces maps
$\Ep(\mathbb{Q})/2\Ep(\mathbb{Q})\to
\Ep(\mathbb{Q}_v)/2\Ep(\mathbb{Q}_v)$
as well as $\mathbb{Q}(S)\to\mathbb{Q}^*_v/{\mathbb{Q}_v^*}^2$ and
$H\to H_v$; they will all be denoted as $\iota_v$. The $2$-Selmer group
$S^2(\Ep/\mathbb{Q})\subset H$ is given by
\[
S^2(\Ep/\mathbb{Q}):=\left\{c=(c_1,c_2,c_3)\in H\;:\; 
\iota_v(c)\in \delta_v(\Ep(\mathbb{Q}_v))\,\forall\, v\in S
\right\}.
\]
It fits in a short exact sequence
\[
0\to \Ep(\mathbb{Q})/2\Ep(\mathbb{Q})
\stackrel{\delta}{\longrightarrow} S^2(\Ep/\mathbb{Q})
\longrightarrow \Sh(\Ep/\mathbb{Q})[2]\to 0
\]
where $\Sh(\Ep/\mathbb{Q})[2]$ is the $2$-torsion in the
Shafarevich-Tate group of $\Ep/\mathbb{Q}$. To ease notation,
for fields $F$ an element $c{F^*}^2\in F^*/{F^*}^2$ will simply
be denoted $c$. A short calculation shows:
\begin{itemize}
    \item $\delta_\infty(\Ep(\mathbb{R}))$ is generated by $(1,-1,-1)$;
    \item $\delta_p(\Ep(\mathbb{Q}_p))$ is generated by
    $(2,-p,-2p),\;(p,-1,-p)$ for $p\neq 2$;
    \item $\delta_2(\Ep(\mathbb{Q}_2))$ has generators
    $(2,-p,-2p),\;(p,-1,-p),\;(5,1,5)$ for $p\neq 2$. Here,
    apart from $2$-torsion one may use the point
    $(\frac{1}{4},\frac18\sqrt{1-16p^2})\in\Ep(\mathbb{Q}_2)$.
\end{itemize}
An immediate consequence is the following.
\begin{proposition}\label{SelQ}
$\dim_{\mathbb{F}_2}\,S^2(\Ep/\mathbb{Q})=
\left\{
\begin{array}{ll}
4 & \text{if }\;p\equiv 1\bmod 8;\\
2 & \text{if }\;p\equiv 3\bmod 8;\\
3 & \text{if }\;p\equiv 5,7\bmod 8.
\end{array}
\right.$
\end{proposition}
\begin{proof}
We only sketch the case $p\equiv 1\bmod 8$, the other cases are
analogous. Since $p\equiv 1\bmod 8$, one has that $-1, 2$ are squares in
$\mathbb{Q}_p$ hence the $\delta_p$-image is generated by
$(1,p,p)$ and $(p,1,p)$. Similarly, the $\delta_2$-image has
generators $(2,-1,-2),\;(1,-1,-1),\;(5,1,5)$. Considering the
$\delta_\infty$-image one observes that the first coordinate
of any $c\in S^2(\Ep/\mathbb{Q})$ is one of $\{1,2,p,2p\}$.
Now consider the possibilities:\\
if $c=(1,a,a)$, then the local conditions at $2,p$ imply $a\in\{\pm 1, \pm p\}$, giving $4$ elements. For $c=(2,a,2a)$ again one finds
$a\in\{\pm 1, \pm p\}$, and the same conclusion holds for
 $c=(p,a,ap)$ and $c=(2p,a,2ap)$.
\end{proof}
Using Monsky's result (predicted by Stephens) asserting that
primes $p\equiv 5,7\bmod 8$ are congruent, the following is a consequence.
\begin{corollary}
For any odd prime $p\not\equiv 1\bmod 8$ one has
\[\Sh(\Ep/\mathbb{Q})[2]=(0)\;\; \textit{and}\;\;
\textit{rank}\,\Ep(\mathbb{Q})=\left\{
\begin{array}{ll}
0 & \textit{if }\; p\equiv 3\bmod 8,\\
1 & \textit{if }\; p\equiv 5,7\bmod 8.
\end{array}\right.
\]
\end{corollary}
\begin{proof}
Since $\text{rank}\,\Ep(\mathbb{Q})=\dim_{\mathbb{F}_2}
\Ep(\mathbb{Q})/2\Ep(\mathbb{Q})-2$, the exact sequence for
$S^2$ implies 
\[\text{rank}\,\Ep(\mathbb{Q})+\dim_{\mathbb{F}_2}\Sh(\Ep/\mathbb{Q})[2]=\dim_{\mathbb{F}_2}\,S^2(\Ep/\mathbb{Q})-2.
\]
Together with Proposition~\ref{SelQ} this implies the corollary for
the primes
$p\equiv 3\bmod 8$. Since Monsky's result shows that $\Ep/\mathbb{Q}$
has positive rank for $p\equiv 5,7\bmod 8$, the corollary also
follows in those cases.
\end{proof}

In the remainder of this section let $K=\mathbb{Q}(\sqrt{p})$.
Over $K$ the curve $\Ep$ is isomorphic to
\[
E\colon y^2=x^3-x.
\]
Moreover $\text{rank}\, E(K)=\text{rank}\, E(\mathbb{Q})+\text{rank}\, \Ep(\mathbb{Q})$, corresponding to the decomposition of
$E(K)\otimes\mathbb{Q}$ into eigenspaces for $\gal(K/\mathbb{Q})$.
It is well known that $\text{rank}\, E(\mathbb{Q})=0$ (equivalently, $1$ is not a congruent number, a fact already stated by Fibonacci in
1225 and a consequence of a result of Fermat; a detailed proof is
also presented in \cite{Barlow}). Hence
\[ \text{rank}\,\Ep(\mathbb{Q}) = \text{rank}\, E(K).\] 
In some cases $p\equiv 1\bmod 8$, the bound $\text{rank}\,E(K)\leq \dim_{\mathbb{F}_2}S^2(E/K)-2$
involving the $2$-Selmer group of $E$ over $K=\mathbb{Q}(\sqrt{p})$
 gives a stronger
bound for $\text{rank}\,\Ep(\mathbb{Q})$ than the one using
$S^2(\Ep/\mathbb{Q})$. In such a case it follows that
$\Sh(\Ep/\mathbb{Q})[2]$ is nontrivial.
\begin{proposition}\label{SelK}
Let $p\equiv 1\bmod 8$ be prime, and let $j\in\mathbb{Z}$ satisfy
$j^2\equiv -1\bmod p$. Then 
$\dim_{\mathbb{F}_2} S^2(E/\mathbb{Q}(\sqrt{p}))=\left\{
\begin{array}{ll}4 & \textit{if}\;\leg{1+j}{p}=1,\\
2 & \textit{if}\;\leg{1+j}{p}=-1.\end{array}
\right.$
\end{proposition}
The result already announced in Bastien's 1915 paper is by the method sketched
above a consequence of this:
\begin{corollary}\label{Densityresult}
Let $p\equiv 1\bmod 8$ be a prime satisfying the
equivalent conditions given in Proposition~\ref{equiv1mod8}.
Then $\textit{rank}\,\Ep(\mathbb{Q})=0$
and $\Sh(\Ep/\mathbb{Q})[2]\cong 
\left(\mathbb{Z}/2\mathbb{Z}\right)^2$. In particular, $p$ is not congruent.

The set of primes discussed here has natural density $1/8$.
\end{corollary}

\begin{proof}
Most of this is immediate from 
Propositions~\ref{SelQ} and \ref{SelK} by
the argument sketched above.
The assertion about the density is 
obtained by applying Chebotar\"ev's 
density theorem (see, e.g. \cite{Lenstra-Stevenhagen}) in the
situation of Proposition~\ref{equiv1mod8}(7).
\end{proof}

\begin{proof} (of Proposition~\ref{SelK}).
The elliptic curve $E$ has good reduction away from $2$, 
so we take $S=\{\mathfrak{p}_2,\mathfrak{q}_2,\infty_1,\infty_2\}$, with $\mathfrak{p}_2,\mathfrak{q}_2$ the primes over $2$ in $K=\mathbb{Q}(\sqrt{p})$ and $\infty_1,\infty_2$ the two real embeddings. 
As with $E^{(p)}/\mathbb{Q}$, set $K(S)=\{x\in K^*/K^{*2}:v(x)\equiv 0\bmod 2\text{ for all }v\notin S\}$ and $H=\{(c_1,c_2,c_3):K(S)\times K(S)\times K(S):c_1c_2c_3=1\}$. The Kummer homomorphism
$\delta\colon E(K)\to H$ is then given by
\[
\delta(a,b)=\left\{\begin{array}{ll}
(a+1,a,a-1) & \text{for } b\neq 0;\\
(2,-1,-2) & \text{for } a=-1;\\
(1,-1,-1) & \text{for } a=0;\\
(2,1,2) & \text{for } a=1.
\end{array}\right.
\]
Regarding the local images $\delta_{v}(E(K_v))$, 
note that the completion $K_v$ equals $\mathbb{Q}_2$ for
$v\in\{\mathfrak{p}_2,\mathfrak{q}_2\}$ and $K_v=\mathbb{R}$
for $v\in\{\infty_1,\infty_2\}$. Observe:
\begin{itemize}
    \item $\delta_v(E(\mathbb{R}))$ is generated by $(1,-1,-1)$ for $v\in\{\infty_1,\infty_2\}$,
    \item $\delta_v(E(\mathbb{Q}_2))$ has generators
    $(2,-1,-2),\;(1,-1,-1),\;(-3,1,-3)$ for $v\in\{\mathfrak{p}_2,\mathfrak{q}_2\}$.
\end{itemize}
The group $K(S)$ fits in an exact sequence
\[
0\to R_S^*/R_S^{*2}\to K(S)\to\cl(R_S)[2]\to 0,
\]
where $R_S$ is the ring of $S$-integers of $K$. We have $\cl(R_S)[2]=0$ since $\cl(R_S)$ is a quotient of $\cl_K$ and $K$ has odd class number. 
As $R_S^{*}$ has rank $3$ and torsion subgroup
$\{\pm 1\}$ it follows that $K(S)$ has $\mathbb{F}_2$-dimension $4$. 
We now describe a convenient basis of $K(S)$. 

Given $c\in K(S)$ and $v\in S$,
write $\im_v(c)$ for the image of $c$ under 
$K(S)\to K_v^*/K_v^{*2}$. 
Note that a fundamental unit of $K$ has negative norm (the narrow Hilbert class field of $K$ has odd degree). 
Choose a fundamental unit $\varepsilon\in\mathcal{O}_K^*$ with $\infty_1(\varepsilon)>0$, 
so $\infty_2(\varepsilon)<0$. 
Since $\varepsilon\overline{\varepsilon}=-1$,
$\im_{\mathfrak{p}}(\varepsilon)=\im_{\mathfrak{q}}(\overline{\varepsilon})$ for $\mathfrak{p}$ and $\mathfrak{q}$ conjugate, and $\im_{\mathfrak{p}_2}(\varepsilon)\subset\langle -1,3\rangle\subset\mathbb{Q}_2^*/\mathbb{Q}_2^{*2}$, precisely one of $\mathfrak{p}_2,\,\mathfrak{q}_2$ is unramified in $K(\sqrt{\varepsilon})$. By interchanging the two if necessary we may and will assume that $\mathfrak{p}_2$ is unramified in $K(\sqrt{\varepsilon})$. Equivalently, this means
$\im_{\mathfrak{p}_2}(\epsilon)\in\{1,-3\}$, and of course
$\im_{\mathfrak{q}_2}(\epsilon)=-\im_{\mathfrak{p}_2}(\epsilon)$. In terms of a R\'edei symbol as in \cite{stevenhagen2018redei}, in fact
$\im_{\mathfrak{p}_2}(\epsilon)=1\Leftrightarrow[p,-1,2]=1$.
\\
Take $k$ any odd multiple of the order of
$[\mathfrak{p}_2]\in\cl_K$ and write $\mathfrak{p}_2^k=(x_2)$ for 
some $x_2\in K$. Multiplying $x_2$ by $\pm \varepsilon$ if necessary we may assume $x_2$ has positive norm and  that $\mathfrak{q}_2$ is unramified in $K(\sqrt{x_2})$, so
$x_2\cdot \overline{x_2}=2^k$ and 
$\im_{\mathfrak{q}_2}(x_2)\in\{1,-3\}$.
These choices yield $K(S)=\langle -1,\varepsilon,x_2,y_2\rangle$, with $y_2=\overline{x_2}$. Moreover $\im_{\infty_m}(x_2)
=\im_{\infty_m}(y_2)$
is independent of $m\in\{1,2\}$ and it
equals the R\'{e}dei symbol $[p,2,-1]$. 
R\'edei reciprocity \cite[Thm.~1.1]{stevenhagen2018redei} implies
in particular $[p,2,-1]=[p,-1,2]$, so 
$\im_{\infty_1}(x_2)=1\Leftrightarrow \im_{\mathfrak{p}_2}(\epsilon)=1$. Again using
R\'edei reciprocity one has, for $j\in\mathbb{Z}$
such that $j^2\equiv -1\bmod p$, the equality
$[p,2,-1]=[2,-1,p]=\leg{1+j}{p}$ (note that this
is exactly the example discussed in
\cite[Section~6]{stevenhagen2018redei}).
For $c$ any of these generators
and $v\in S$, this means that the following tables give $\im_v(c)\in K_v^*/K_v^{*2}$.

\vspace{\baselineskip}\noindent
 Case $\leg{1+j}{p}=1$ (with $j\in\mathbb{Z}$ such that $j^2\equiv -1\bmod p$):
\begin{equation*}
\begin{array}{c|c c c c}
 & \mathfrak{p}_2 & \mathfrak{q}_2 & \infty_1 & \infty_2 \\
\hline
-1 & -1 & -1 & -1 & -1 \\
\varepsilon & \cellcolor[gray]{0.8}1 & \cellcolor[gray]{0.8}-1 & 1 & -1 \\
x_2 & \cellcolor[gray]{0.65}2a & \cellcolor[gray]{0.65}a & \cellcolor[gray]{0.8}1 & \cellcolor[gray]{0.8}1 \\
y_2 & \cellcolor[gray]{0.65}a & \cellcolor[gray]{0.65}2a & \cellcolor[gray]{0.8}1 & \cellcolor[gray]{0.8}1 \\
\end{array}
\end{equation*}
Here $a=\im_{\mathfrak{q}_2}(x_2)\in\langle-3\rangle\subset\mathbb{Q}_2^*/\mathbb{Q}_2^{*2}$.

To compute $S^2(E/K)=\{c\in H:\iota_v(c)\in\delta_v(E(K_v))\;\forall\, v\in S\}$
in this situation,
write $S^2(E/K)=\delta(E(K)[2])\oplus A$ with
\[
A=\{c\in S^2(E/K)\;:\;\iota_{\mathfrak{p}_2}(c)\in\langle (-3,1,-3)\rangle\}.
\]
Let $c=(c_1,c_2,c_3)\in A$. As $c_1$ is totally positive 
it follows that $c_1\in\langle x_2,y_2\rangle$.  Considering both possibilities for $\im_{\mathfrak{p}_2}(x_2)$ one finds that
$\im_{\mathfrak{p}_2}(c_1)\subset\langle -3\rangle$ implies  $c_1\in\langle y_2\rangle$. 
The $\mathfrak{p}_2$-adic and $\mathfrak{q}_2$-adic images 
of $c_2$ result in $c_2\in\langle\varepsilon\rangle$. 
Hence $A$ consists of at most four elements. One checks that 
$(y_2,1,y_2),(1,\varepsilon,\varepsilon)\in A$ 
regardless of $a=\im_{\mathfrak{q}_2}(x_2)\in\{1,-3\}$. 
As a consequence  $\dim_{\mathbb{F}_2} A=2$ and therefore $\dim_{\mathbb{F}_2}S^2(E/K)=4$ in this situation.\\

\vspace{\baselineskip}\noindent
The other case is $\leg{1+j}{p}=-1$ (again $j^2\equiv -1\bmod p$). Here one has
\begin{equation*}
\begin{array}{c|c c c c}
 & \mathfrak{p}_2 & \mathfrak{q}_2 & \infty_1 & \infty_2 \\
\hline
-1 & -1 & -1 & -1 & -1 \\
\varepsilon & \cellcolor[gray]{0.8}-3 & \cellcolor[gray]{0.8}3 & 1 & -1 \\
x_2 & \cellcolor[gray]{0.65}2a & \cellcolor[gray]{0.65}a & \cellcolor[gray]{0.8}-1 & \cellcolor[gray]{0.8}-1 \\
y_2 & \cellcolor[gray]{0.65}a & \cellcolor[gray]{0.65}2a& \cellcolor[gray]{0.8}-1 & \cellcolor[gray]{0.8}-1 \\
\end{array}
\end{equation*}
with $a=\im_{\mathfrak{q}_2}(x_2)\in\langle-3\rangle\subset\mathbb{Q}_2^*/\mathbb{Q}_2^{*2}$.
As before, one uses the decomposition
$S^2(E/K)=\delta(E(K)[2])\oplus A$ in which one takes the subgroup
$
A=\{c\in S^2(E/K)\;:\;\iota_{\mathfrak{p}_2}(c)\in\langle (-3,1,-3)\rangle\}$.

Let $c=(c_1,c_2,c_3)\in A$. Considering the
$\mathfrak{p}_2$-adic image shows that either
$c_2=y_2$ and $a=1$ which is inconsistent with the
possibilities for $\iota_{\mathfrak{q}_2}(c)$,
or $c_2=\varepsilon y_2$ and $a=-3$ (again, inconsistent
with $\iota_{\mathfrak{q}_2}(c)$), or $c_2=1$.
Since $c_1$ is
totally positive, one has $c_1\in\{1,-x_2,-y_2,x_2y_2\}$.
The possibility $c_1=-x_2$ is excluded by considering
the $\mathfrak{p}_2$-image, and in the same way $c_1=-y_2$
is impossible because of the $\mathfrak{q}_2$-image.
Since also $c_1=x_2y_2$ is not compatible with $c\in A$,
it follows that $c_1=1$ and therefore $A=(0)$, so  $\dim_{\mathbb{F}_2}S^2(E/K)=2$, completing the proof.
\end{proof}
\section{Examples}\label{Sect3}
A remarkable result by J.B.~Tunnell \cite{Tunnell} (here restricted to odd
integers $n$) states, in a version proposed by N.D.~Elkies,
that
\[
\begin{array}{c}
n\;\textit{is congruent}\\
\Downarrow\\
2x^2+y^2+8z^2=n\; \textit{has as many solutions in}\;\mathbb{Z}^3\;
\textit{with}\; 2|z \;\textit{as with}\; 2\nmid z.
\end{array}
\]
This gives an easy method for showing that certain integers are {\em not} congruent. Applying this to the primes $p\equiv 1\bmod 8$
below some bound $b$ such that $p$ does not satisfy the conditions given in Proposition~\ref{equiv1mod8} one obtains (we used Magma \cite{Magma} for this and further calculations)  the following table.
\[
\begin{array}{rcccccc}
b: &2000&4000&6000&8000&10000&12000\\
\pi(b): &303 &550&783&1007&1229&1438\\
\#p\equiv1\bmod 8 \;\&\; \leg{1+i}{p}=1: &30&62&93&116&146&172\\
\#p \text{noncongruent:} &19&41&62&78&99&122
\end{array}
\]
As a consequence of \cite[Thm.~1.1]{Wang} the primes $p$ discussed here
(i.e., $p\equiv 1\bmod 8$ and $\leg{1+i}{p}=1$ and $p$ is noncongruent) have the property
$\Sh(\Ep/\mathbb{Q})[2^\infty]\supsetneq 
\left(\mathbb{Z}/2\mathbb{Z}\right)^2$. In particular this implies
that $\Sh(\Ep/\mathbb{Q})$ contains elements of order $4$. For
the primes contributing to the table (so $p<12000$, the set looks like $\{113,\,337,\,409, \ldots\ldots,\,11897,\,11969\}$) a simple
Magma test confirms this as well:
\begin{verbatim}
  E:=EllipticCurve([-p^2,0]);
  MordellWeilShaInformation(E : ShaInfo);
\end{verbatim}

The table suggests that at least $1/12$-th of the set of all primes 
has the property described here. If the converse of Tunnell's theorem
holds (as would be a corollary of the Birch and Swinnerton-Dyer conjecture), this would mean that at most $1/6$-th of the set
of primes $p\equiv 1\bmod 8$ is congruent.
A Magma test 
 yields the following data.
\[
\begin{array}{rccccccc}
\text{bound}\;b: & 2000 & 4000 & 6000 & 8000 & 10000 & 12000&14000\\
\# \leq b,\;\equiv 1\bmod 8: &68 & 129 & 186 & 243 & 295 & 341&400\\
\# p\;\text{congruent}: &11 & 21 & 31 & 38 & 47 & 50 &58
\end{array}
\]
The $58$ primes $\{41,\,137,\,257,\,\ldots\ldots,9377,\ldots\ldots,13513,13841,13921\}$
here all result in $\text{rank}\Ep(\mathbb{Q})=2$ and hence
$\Sh(\Ep/\mathbb{Q})[2]=(0)$. Indeed, Magma quickly finds two
independent points except in one case: for $p=9377$ a point of
infinite order with $x=-\frac{6635776}{4225}$ is easily found.
To obtain a second, independent rational point one needs to
increase the {\tt Effort} parameter in {\tt MordellWeilShaInformation}, resulting in the points with
\[
x=\frac{9377\cdot(46111436236957655053256122338300576234143337)^2}{28100967414057580762568605652421621908428616^2}.
\]
Not surprisingly, our data is consistent with the list of  congruent numbers $n<10000$ available from \cite{OEIS}.
It is conceivable that using e.g. higher descents, a density result for primes
$p\equiv 1\bmod 8$ satisfying
$(\mathbb{Z}/4\mathbb{Z})^2\subset 
\Sh(\Ep/\mathbb{Q})$ may be obtained.

In contrast, the question whether or not infinitely
many primes $p\equiv 1\bmod 8$ are congruent numbers seems much harder.
We finish this paper by showing that
a special case of Bouniakowsky's
conjecture implies the existence of infinitely many such primes.
Recall that Bouniakowsky's conjecture
from 1854 (see \cite[p.~328]{Buniakowsky}) states that if $g(x)\in\mathbb{Z}[x]$
is irreducible, has leading coefficient $>0$, and the gcd of all values
$\{g(m)\;:\;m\in\mathbb{Z}\}$ equals $1$,
then $g(m)$ is claimed to be a prime
number for infinitely many $m\in\mathbb{Z}$. Note that this is a
very early special case of Schinzel's Hypothesis~H
formulated in \cite[p.~188]{Schinzel-Sierpinski}.
\begin{proposition}\label{examples1mod8}
The polynomial $f(x):=8x^4+16x^3+12x^2+4x+1$
satisfies the conditions in Bouniakowsky's conjecture, so $f(\mathbb{Z})$ is
supposed to contain infinitely many
prime numbers.

If $n\in f(\mathbb{Z})\setminus\{1\}$, then $n\equiv 1\bmod 8$ and $n$ is a congruent number.
\end{proposition}
\begin{proof}
It is not difficult to verify that $f(x)\in\mathbb{Z}[x]$ is irreducible (in fact, up to a a factor $2$ it equals the $8$-th
cyclotomic polynomial evaluated in $2x+1$). Clearly the leading
coefficient is positive, and since $f(0)=1$ the gcd of the
values $f(k)$ equals $1$. This shows the first part of the
proposition.

For integers $k$, note that $f(k)\equiv 4(k^2+k)+1\bmod 8= 1\bmod 8$.
Also, $2f(k)=(2k+1)^4+1$ which implies $f(k)\geq 1$ whenever $k\in\mathbb{Z}$
and $f(k)=1\Leftrightarrow k\in\{-1,0\}$ (assuming $k$ is real).
Hence to complete the proof, take $n=f(k)$ with $k\in\mathbb{Z}\setminus\{-1,0\}$, so that $n\geq f(1)=41$ (because
$f(x)=f(-x-1)$ and $x\mapsto f(x)$ increases for $x>-\frac12$). Then
\[
\begin{array}{l}
a:=2n\cdot(2k+1)^2,\\
b:=16\cdot(k^2+k)^2(2k^2+2k+1)^2=n^2-(2k+1)^4,\\
c:=n^2+(2k+1)^4
\end{array}
\]
are positive integers satisfying $a^2+b^2=c^2$, hence they arise as
lengths of a right-angled triangle with area $n$ times a square.
This implies that $n$ is a congruent number.
\end{proof}
\begin{remark}\label{DensityRemark}
Of the $58$ prime numbers $p\equiv 1\bmod 8$ with $p<14000$ that are
congruent, $4$ are in $f(\mathbb{Z})$. Note that $f(\mathbb{Z})$
contains only $5$ integers in the interval $[2, 14000]$.
A special case of a conjecture of P.T.~Bateman and R.A.~Horn predicts the number of
integers $1\leq k\leq b$ such that $f(k)$ is prime: it should grow
$\sim \frac{C(f)}{4}b/\ln(b)$ with $C(f)=\prod_p\frac{p-\omega(p)}{p-1}$ and $\omega(p)=4$ for
$p\equiv 1\bmod 8$ while $\omega(p)=0$ otherwise. An approximation $C(f)\approx 5.358$ one obtains using \cite{Shanks}.

Since $f(\frac{u}{v})=\left((2\frac{u}{v}+1)^4+1\right)/2=
\frac1{v^4}\left((2u+v)^4+v^4\right)/2$, the proof of
Proposition~\ref{examples1mod8} shows that integers represented
by the binary form 
\[\left((2u+v)^4+v^4\right)/2 = 8u^4+16u^3v+12u^2v^2+4uv^3+v^4
\]
(except for $u=0$ and for $u=-v$) are congruent numbers.
Such numbers are $\equiv m^4\bmod 8$, so $\equiv 1\bmod 8$ whenever
$m$ is odd.
Below the bound $14000$ this yields $9$ prime congruent numbers,
namely 
\[41,\,
313,\,
353,\,
1201,\,
3593,\,
4481,\,
7321,\,
8521,\,
10601.
\]
Applying \cite[Thm.~1]{Stewart-Top} to the binary form given here
yields a constant $C>0$ such that up to any bound $B$ it attains at least $C\sqrt{B}$ square-free values; by construction these
are congruent numbers $\equiv 1\bmod 8$ since for $v$ even,
the number would be divisible by $8$ hence it would not be
square-free. It seems that this is slightly stronger than what
would be obtained starting from $(2uv)^2+(u^2-v^2)^2=(u^2+v^2)^2$,
i.e., from the binary form $uv(u+v)(u-v)$.
\end{remark}
%\bibliographystyle{plain}
%\bibliography{ref.bib}
\printbibliography

\end{document}